\documentclass[11pt]{amsart}
\usepackage{amscd,bbm,dsfont}
\usepackage{amsmath,amssymb,mathrsfs,amsthm}
\usepackage{amsfonts}
\theoremstyle{plain}
\newtheorem{lemma}{Lemma}[section]
\newtheorem{theorem}{Theorem}[section]

\newtheorem{corollary}{Corollary}[section]
\theoremstyle{definition}

\theoremstyle{remark}
\newtheorem{remark}{Remark}[section]
\numberwithin{equation}{section}

\newcommand{\abs}[1]{\left\vert#1\right\vert}
\newcommand{\p}{\partial}

\newcommand{\dist}{\mathrm{dist}}

\newcommand{\ls}{\leqslant}
\newcommand{\gs}{\geqslant}
\makeatletter

\newcommand{\Rmnum}[1]{\mathrm{\expandafter\@slowromancap\romannumeral#1@}}
\makeatother
\begin{document}

\title[Sharp estimate of lower bound for the first eigenvalue]
{Sharp estimate of lower bound for the first \\ eigenvalue in
the Laplacian operator \\ on compact Riemannian manifolds}
%\title[A estimate of lower bound for the first eigenvalue]
%{A estimate of lower bound for the first \\ eigenvalue in
%the Laplacian operator \\ on compact Riemannian manifolds}

%\author[H. Yuan]{Hairong Yuan\,$^1$}

\author[Yue He]{Yue He}

\address{Yue He,
Institute of Mathematics, School of Mathematics Sciences, Nanjing
Normal University, Nanjing 210023, China} \email{heyueyn@163.com
\& heyue@njnu.edu.cn.}

\thanks{Partially supported
%in part
by NNSF grant of P. R. China: No.\,11171158.}
%and Natural Science Foundation of
%Jiangsu Education Commission of P. R. China:
%No.\,$\cdots\cdots\cdots$.}


%53C40

\date{\today}

\begin{abstract}
The aim of this paper is give a simple proof of some results in
\cite{Jun Ling-2006-IJM} and \cite{JunLing-2007-AGAG},
which are very deep studies in the sharp lower bound of the first
%nontrivial
eigenvalue in the Laplacian operator on compact Riemannian manifolds with
nonnegative Ricci curvature. We also get a result about lower bound of the first
Neumann eigenvalue in a special case. Indeed, our estimate of lower bound
in the this case is optimal. Although the methods used in here
%this paper
due to \cite{Jun Ling-2006-IJM} (or \cite{JunLing-2007-AGAG})
on the whole, to some extent we
%deal with, handle
can tackle the singularity of test
%some
functions and also simplify greatly
much calculation in these references. Maybe this provides another way
%a new way
to estimate eigenvalues.
\end{abstract}

\maketitle

%%%---------------------------------------------------------------------------------------------------------

%brief

\section{Introduction}

Suppose $(M,g)$ is an n-dimensional compact Riemannian manifold with Ricci curvature satisfies
\begin{equation}\label{Ricci-curvature}
\mathrm{Ric}(M)\gs (n-1)K
\end{equation}
for some nonnegative constant $K$.

Unlike upper bound estimates, lower bound estimates for an eigenvalue
is difficult to obtain. The study on the lower bound of the first positive
eigenvalue in the Laplacian operator on compact Riemannian manifolds,
can trace its history to a long time ago. In the meanwhile, there are many
works in this area. Among these works, the results of Li
(\cite{Li-1979}, \cite{Li-1982}), Li--Yau (\cite{Li-Yau-1980},
\cite{Li-Yau-1983}), Zhong--Yang \cite{Zhong-Yang-1983}, Yang
\cite{DaGang Yang-1999}, Ling \cite{Jun
Ling-2006-IJM}--\cite{Jun Ling-2007-MN}, Ling--Lu
\cite{Ling-Lu-2010}, Shi--Zhang \cite{Shi-Zhang-2007}, Qian--Zhang--Zhu
\cite{Qian-Zhang-Zhu-2012}, Andrews--Ni
\cite{Andrews-Ni-1111.4967}, and Andrews--Clutterbuck
\cite{Andrews-Clutterbuck-1204.5079}, etc., are all very well
known.
It is difficult to describe all references in this field. %area.
So we just outline a portion of
%some
important works.

%In 1958, Lichnerowicz \cite{Lichnerowicz-1958}
%(see also \cite{Obata-1962}) has shown that
%the first positive eigenvalue $\lambda_1$ has a lower bound
%\begin{equation}\label{Lichnerowicz}
%\lambda_1\gs nK
%\end{equation}
First of all, we state the following lower bound
estimate of the first eigenvalue, which is due to Lichnerowicz 1958
\cite{Lichnerowicz-1958} (also see Obata \cite{Obata-1962}) when M is a
compact manifold without boundary. Under the same
assumption \eqref{Ricci-curvature}, Escobar \cite{Escobar} proved that
if a compact manifold has a weakly convex boundary, the first nonzero Neumann eigenvalue
of M has the following lower bound \eqref{Lichnerowicz-estimate} as well.
%For completeness sake, a proof is enclosed.
\begin{theorem}\label{thm-Lichnerowicz}
$($see $\mathbf{Ling}$ \cite{DaGang Yang-1999}$)$ Assume that $\mathrm{Ric}(M)\gs
(n-1)K >0$. Let $\lambda_1$ be the first positive eigenvalue on
$M$ (with either Dirichlet or Neumann boundary condition if $\p
M\neq\emptyset$). If $\p M\neq\emptyset$, we also assume that $\p
M$ is of nonnegative mean curvature $\mathrm{tr} S\gs 0$ if
$\lambda_1$ is the first Dirichlet eigenvalue and $\p M$ is of
nonnegative definite second fundamental form $S\gs 0$ if
$\lambda_1$ is the first Neumann eigenvalue. Then
\begin{equation}\label{Lichnerowicz-estimate}
\lambda_1\gs nK
\end{equation}
\end{theorem}
%satisfactory
This estimate provide no information when the above constant K vanishes. In
such case, Li-Yau [5] and Zhong-Yang [12] provided another lower bound.

It is an interesting problem to find a unified lower bound of the first non-zero
eigenvalue $\lambda_1$ in terms of the lower bound $(n-1)K$ of the Ricci curvature and
the diameter $d$, the inscribed radius $r$ and other geometric quantities, which do not
vanish as $K$ vanishes, of the manifold with positive Ricci curvature.

Later on, the maximum principle method which is rather different to
that before, was first used by Li 1979 \cite{Li-1979} in proving
eigenvalue estimates for compact manifolds.
%Now we briefly introduce the major steps of this method
%to establish the eigenvalue estimate in the following.
%
%Firstly, one need construct some suitable test function which is especially
%useful in this method. Secondly, the maximum principle can be applied to
%%used to
%such test function for deriving
%%deduce
%the gradient estimate which is relevant to the first eigenfunction.
%In particular, it is the key step to establish the gradient estimate for deriving the eigenvalue estimate.
%Finally, the lower bound of the first eigenvalue is obtained by integrating such
%gradient estimate along the shortest curve which is joining two points.
From that time on, this method was then refined and used by many authors
(e.g., Li--Yau \cite{Li-Yau-1980}, Zhong--Yang \cite{Zhong-Yang-1983},
Yang \cite{DaGang Yang-1999}, etc.) for obtaining sharper eigenvalue estimates.

Soon after, using a improved maximum principle method, Li--Yau 1980
\cite{Li-Yau-1980} derived the following beautiful result in the case when
$K=0$.
\begin{theorem}\label{Li-Yau-1980 Main Theorem}
$(\mathbf{Li-Yau})$. Let $M$ be a compact Riemannian manifold, $\p
M=\phi$, $\mathrm{Ric}(M)\gs 0$, then $\lambda_1\gs
\frac{\pi^2}{4d^2}$, where $d=\mathrm{diam} (M)$ is the diameter
of $M$.
\end{theorem}

The above result was improved by Li 1982 \cite{Li-1982} to
$\lambda_1\gs \frac{\pi^2}{2d^2}$ in the case when $K=0$. At one time
Li had also conjectured that the first positive eigenvalue should satisfy
\begin{equation}\label{Li-conjectured-estimate}
\lambda_1\gs \frac{\pi^2}{d^2}+(n-1)K.
\end{equation}
This conjecture greatly motivate many related studies in this area.
%We will state those main results here.
It is necessary to mention those main results in the following.

Firstly, recall that the well-known Bonnet--Myers Theorem:
\begin{theorem}\label{Bonnet-Myers-thm}
$(\mathbf{Bonnet-Myers})$ Suppose that $M$ is a $n$-dimensional
complete Riemanian manifold with Ricci curvature bounded below by
$(n-1)K$ $(K>0)$. Then $M$ is compact, and its diameter $d(M)$
satisfies the following estimate
\begin{equation}\label{Bonnet-Myers-diameter-estimate}
d(M)\ls \frac{\pi}{\sqrt{K}}.
\end{equation}
\end{theorem}
Combining \eqref{Li-conjectured-estimate} with
\eqref{Bonnet-Myers-diameter-estimate}, we can deduce that
\eqref{Lichnerowicz-estimate} again. So
\eqref{Li-conjectured-estimate} is %often
usually regarded as
%is considered to
%is thought to
the sharp lower bound on $\lambda_1$ in terms of diameter for
manifolds with Ricci curvature satisfies \eqref{Ricci-curvature}.
Obviously, the optimal estimate to lower bound of the first
eigenvalue is perfect and powerful. It seems that any further
progress requires a refined gradient estimate which is relevant to
the first eigenfunction.

By sharpening Li--Yau's method and giving a more
delicate estimate, Zhong--Yang 1983 \cite{Zhong-Yang-1983}
improved this to the sharp estimate $\lambda_1\gs
\frac{\pi^2}{d^2}$ in the case when $K=0$.
%Eigenvalue estimate are beautiful and powerful, especially when they are optimal.
We now show Zhong--Yang's remarkable result as follows.
%as follows.
\begin{theorem}\label{Zhong-Yang-1983 Main Theorem}
$(\mathbf{Zhong-Yang})$. Let $M$ be a compact Riemannian manifold
without boundary, with nonnegative Ricci curvature and let $d$ be
the diameter of $M$. Then,
\begin{equation}\label{11-5-23-21-02}
\lambda_1\gs \frac{\pi^2}{d^2}.
\end{equation}
\end{theorem}

%We also remark here that a proof of
Next, we also remark here that the attempt to prove
the so-called Li's conjecture would unify Yang--Zhong's estimates
%for manifolds with nonnegative Ricci curvature
with Lichnerowicz's estimate. Several previous
%attempts
efforts to prove \eqref{Li-conjectured-estimate}
%such inequalities
have been made, particularly towards improving
%proving
inequalities of the form
\begin{equation}\label{esti-12-9-29-9-55}
\lambda_1\gs \frac{\pi^2}{d^2}+\alpha(n-1)K
\end{equation}
for some constant $\alpha$.
These include works of Yang
\cite{DaGang Yang-1999}, Ling \cite{Jun Ling-2006-IJM}--\cite{Jun Ling-2007-MN}, Ling--Lu \cite{Ling-Lu-2010},
%the latter showing that $\alpha=\frac{34}{100}$ holds.
%However, these are all superseded by the result of
Shi--Zhang \cite{Shi-Zhang-2007}, Qian--Zhang--Zhu \cite{Qian-Zhang-Zhu-2012},
Ni \cite{Lei Ni-1107.2351}, Andrews--Clutterbuck \cite{Andrews-Clutterbuck-2011-JAMS}
and \cite{Andrews-Clutterbuck-1204.5079}, also Andrews--Ni \cite{Andrews-Ni-1111.4967}, etc.
%which prove
%\[
%\lambda_1(M)\gs\sup_{s\in(0,1)}\Big\{4s(1-s)\frac{\pi^2}{d^2}+s(n-1)K\Big\},
%\]
%so in particular
%\begin{equation}\label{esti-12-7-30}
%\lambda_1\gs \frac{\pi^2}{d^2}+\frac{(n-1)K}{2}
%\end{equation}
%by taking $s=\frac{1}{2}$.
Now we proceed to state briefly some of these wroks in the following.

In a recent paper, following the similar methods in
%all of the above works,
several previous works, but constructing a more complicated test function,
Yang 1999 \cite{DaGang Yang-1999} has made a certain progress
%partially proved
in Li's conjecture, as shown by the following results.
\begin{theorem}\label{thm-DaGang-Yang-1}
%$($see \cite{DaGang Yang-1999}$)$
$(\mathbf{Yang})$ Let $M^n$ be a closed Riemannian manifold with
$\mathrm{Ric}(M^n)\gs (n-1)K\gs 0$ and diameter $d$. Then the
first positive eigenvalue $\lambda_1$ on $M$ satisfies the lower
bound
\[
\lambda_1\gs \frac{\pi^2}{d^2}+\frac{(n-1)K}{4}.
\]
\end{theorem}

\begin{theorem}\label{thm-DaGang-Yang-2}
%$($see \cite{DaGang Yang-1999}$)$
$(\mathbf{Yang})$ Let $M^n$ be a compact manifold with nonempty
boundary and with $\mathrm{Ric}(M^n)\gs (n-1)K\gs 0$.

$(\textrm{a})$\,\,Assume that the boundary $\p M$ is weakly convex, that is, the
second fundamental form with respect to the outward normal is
nonnegative. Then the first positive Neumann eigenvalue
$\lambda_1$ on $M^n$ satisfies the same lower bound
\eqref{Lichnerowicz-estimate}.

$(\textrm{b})$\,\,Assume that the mean curvature with respect to the outward
normal of the boundary $\p M$ is nonnegative. Then the first
positive Dirichlet eigenvalue $\lambda_1$ on $M^n$ satisfies the
lower bound estimate
\[
\lambda_1\gs \frac{1}{4}\big[\frac{\pi^2}{r^2}+(n-1)K\big].
\]
where $r$ is the inscribed radius for $M^n$.
\end{theorem}

In order to improve the known results above via the maximum principle method,
one need to construct suitable test functions where detailed technical
work is essential. In more recent two papers, Ling 2006
\cite{Jun Ling-2006-IJM} and 2007 \cite{JunLing-2007-AGAG} give
some new estimates on the lower bound and partially improve the
lower bound above. The main results of those references
%\cite{Jun Ling-2006-IJM} and \cite{JunLing-2007-AGAG}
are the following three theorems.
\begin{theorem}\label{thm-12-9-26-13a-15}
$(\mathbf{Ling})$ If $(M, g)$ is an n-dimensional compact Riemannian manifold
with boundary. Suppose that Ricci curvature $\mathrm{Ric}(M)$ of $M$ is bounded below
by $(n-1)K$ for some constant $K>0$
\[
\mathrm{Ric}(M)\gs(n-1)K
\]
and that the mean curvature of the boundary $\p M$ with respect to the outward
normal is nonnegative, then the first Dirichlet eigenvalue $\lambda_1$ of the Laplacian
$\Delta$ of $M$ has the following lower bound
\[
\lambda_1\gs\frac{\pi^2}{\tilde{d}^2}+\frac{1}{2}(n-1)K,
\]
where $\tilde{d}$ is the diameter of the largest interior ball in $M$,
that is, $d =2\sup_{x\in M}\{\dist(x,\p M)\}$.
\end{theorem}
\begin{theorem}\label{thm-Jun-Ling-2007-AGAG-1}
$(\mathbf{Ling})$ If M is an n-dimensional, compact Riemannian
manifold that has an empty or nonempty boundary whose second
fundamental form is nonnegative with respect to the outward normal
$($i.e., weakly convex$)$. Suppose that Ricci curvature
$\mathrm{Ric}(M)$ has a lower bound $(n-1)K$ for some constant
$K>0$, that is
\[
\mathrm{Ric}(M)\gs (n-1)K>0.
\]
Then the first non-zero $($closed or Neumann, which applies$)$
eigenvalue $\lambda_1$ of the Lalacian  on $M$ has the following
lower bound
\[
\lambda_1\gs\frac{\pi^2}{d^2}+\frac{3}{8}(n-1)K\quad\hbox{for}\quad n=2
\]
and
\[
\lambda_1\gs\frac{\pi^2}{d^2}+\frac{31}{100}(n-1)K\quad\hbox{for}\quad n\gs 3,
\]
where where $d$ is the diameter of $M$.
\end{theorem}
\begin{theorem}\label{thm-Jun-Ling-2007-AGAG-2}
$(\mathbf{Ling})$ Under the conditions as in Theorem
\ref{thm-Jun-Ling-2007-AGAG-1}, if the manifold M has the symmetry that
the minimum of the first eigenfunction is the negative of the maximum,
i.e., $k=1$ in \eqref{11-5-16-18-30}, then
%under the conditions in Theorem \ref{thm-Jun-Ling-2007-AGAG-1}
the first nonzero (closed or Neumann, which applies)
eigenvalue $\lambda_1$
%has the following lower bound
satisfies \eqref{esti-12-9-29-9-55} with $\alpha=\frac{1}{2}$.
%\[
%\lambda_1\gs\frac{\pi^2}{d^2}+\frac{1}{2}(n-1)K.
%\]
\end{theorem}

However, these are all
%entirely %largely %superseded %exceeded
entirely updated by the results of Shi--Zhang 2007 \cite{Shi-Zhang-2007},
Qian--Zhang--Zhu \cite{Qian-Zhang-Zhu-2012}, as well as Andrews--Clutterbuck
\cite{Andrews-Clutterbuck-1204.5079}. More precisely, Shi--Zhang
\cite{Shi-Zhang-2007} give the following result via using very differential method.
\begin{theorem}\label{ye-zhang-thm-12-7-13A-17-26}
%$($see \cite{Shi-Zhang-2007}$)$
$(\mathbf{Shi-Zhang})$ Let M be a compact $n$-dimensional
Riemannian manifold without boundary $($or with convex boundary$)$
and $\mathrm{Ric}(M)\gs (n-1)K$. Then its first non-zero
$($Neumann$)$ eigenvalue $\lambda_1(M)$ satisfies
\begin{equation}\label{est-ye-zhang-zhu}
\lambda_1(M)\gs 4s(1-s)\frac{\pi^2}{d^2}+s(n-1)K\quad\hbox{for
all}\quad s\in (0,1),
\end{equation}
where $d$ is the diameter of $M$.
\end{theorem}

Following the same argument in Shi--Zhang \cite{Shi-Zhang-2007},
Qian--Zhang--Zhu \cite{Qian-Zhang-Zhu-2012} generalize this result
to the case when $M$ is a Alexandrov space.
\begin{theorem}\label{zhang-zhu-thm-12-7-13A-17-26}
%$($see \cite{Qian-Zhang-Zhu-2012}$)$
$(\mathbf{Qian-Zhang-Zhu})$ Let M be a compact $n(\gs 2)$-dimensional
Alexandrov space without boundary and $\mathrm{Ric}(M)\gs (n-1)K$.
Then its first non-zero eigenvalue $\lambda_1(M)$ satisfies
\eqref{est-ye-zhang-zhu}, where $d$ is the diameter of $M$.
\end{theorem}
Qian--Zhang--Zhu \cite{Qian-Zhang-Zhu-2012} also gives the following remarks.
\begin{remark}\label{zhang-zhu-remark}
(1) If let $s=\frac{1}{2}$, \eqref{zhang-zhu-thm-12-7-13A-17-26}
becomes \eqref{esti-12-9-29-9-55} with $\alpha=\frac{1}{2}$.
%\[
%\lambda_1(M)\gs \frac{\pi^2}{d^2}+\frac{1}{2}(n-1)K.
%\]
This improves Chen--Wang's result \cite{MuFa Chen-FengYu
Wang-1994}--\cite{MuFa Chen-FengYu Wang-1997} in both $K>0$ and
$K<0$. It also improves Ling's recent results in
\cite{JunLing-2007-AGAG}.

(2) If $K>0$,
%Peter Li conjectures that
%\[
%\lambda_1(M)\gs \frac{\pi^2}{d^2}+(n-1)K.
%\]
Theorem \ref{zhang-zhu-thm-12-7-13A-17-26} implies that ¦Ë
\[
\lambda_1(M)\gs \frac{3}{4}\big[\frac{\pi^2}{d^2}+(n-1)K\big].
\]

(3) If $n\ls 5$ and $K>0$, by choosing some suitable constant $s$,
Qian--Zhang--Zhu \cite{Qian-Zhang-Zhu-2012} also get the following estimate
%we have
\[
\lambda_1(M)\gs
\frac{\pi^2}{d^2}+\frac{1}{2}(n-1)K+\frac{(n-1)^2k^2d^2}{16\pi^2}.
\]
\end{remark}

Recently Andrews--Clutterbuck \cite{Andrews-Clutterbuck-1204.5079} also
get \eqref{esti-12-9-29-9-55} with $\alpha=\frac{1}{2}$.
Their contribution is the rather simple proof using
the long-time behavior of the heat equation
%(a method which was also central in their work on the fundamental gap
%conjecture \cite{Andrews-Clutterbuck-2011-JAMS}, and which has
%also been employed successfully in \cite{Lei Ni-1107.2351})
%which seems considerably easier than the previously available arguments.
which is very likely much easier than the formerly available arguments.
In particular, any trouble
%the complications
arising in previous works from possible asymmetry of the first eigenfunction
is avoided in their argument. A similar argument proving the sharp lower bound for
$\lambda_1$ on a Bakry--Emery manifold had appeared
%was present %may be found
in Andrews--Ni's work \cite{Andrews-Ni-1111.4967}.
Meanwhile, Andrews--Clutterbuck \cite{Andrews-Clutterbuck-1204.5079}
also show that the inequality with $\alpha=\frac{1}{2}$ is the best
possible constant of this kind estimate, in other words,
%In that way, %and in particular
the Li conjecture is false. We refer the interested reader to consult
\cite{Andrews-Clutterbuck-1204.5079} for some details.

%where they proved that the first positive eigenvalue satisfies
%$\lambda_1\gs \pi^2/d^2$ for closed manifolds with nonnegative Ricci curvature.
Finally, notice that for manifolds with small diameter, Theorem
\ref{thm-DaGang-Yang-1}--\ref{zhang-zhu-thm-12-7-13A-17-26} is better than
the estimate \eqref{Lichnerowicz-estimate}.
%by Lichnerowicz.
Therefore these results generalize Theorem \ref{Zhong-Yang-1983 Main Theorem}. For
%efforts
more information in this direction, we refer to the
excellent surveys by Ling--Lu \cite{Ling-Lu-2010}, Qian--Zhang--Zhu
\cite{Qian-Zhang-Zhu-2012}, Ni \cite{Lei Ni-1107.2351},
also Andrews--Clutterbuck \cite{Andrews-Clutterbuck-1204.5079}, and so
on, for further results
%development %statement %discussion
of eigenvalue estimate and all the relevant references therein.
%its references history.
To sum up, with the rapid development of geometric analysis, eigenvalue
estimate on this stage is getting more and more important.
%eigenvalue estimate is playing more and more important part in geometric analysis and PDE theory.

In present work we give a simple proof of Theorem \ref{thm-12-9-26-13a-15}
and \ref{thm-Jun-Ling-2007-AGAG-2}, and also get the following result.
\begin{theorem}\label{thm-12-9-27-8-6}
Under the assumptions as in Theorem \ref{thm-12-9-26-13a-15}, if the manifold
M has the symmetry that the minimum of the first eigenfunction is the negative
of the maximum, i.e., $k=1$ in \eqref{11-5-16-18-30}, then the first
%non-trivial
nonzero Dirichlet eigenvalue $\lambda_1$ of the Laplacian on $M$
%has the following lower bound
satisfies \eqref{esti-12-9-29-9-55} with $\alpha=\frac{1}{2}$.
%\[
%\lambda_1\gs \frac{\pi^2}{d^2}+\frac{(n-1)K}{2}.
%\]
\end{theorem}

%Even though
It is a key of this paper to constructing some suitable test function,
even though we mainly use Zhong--Yang's original approach
%method
\cite{Zhong-Yang-1983} in our proof.
%especially useful
Our argument is also based on several
%early
previous works, e.g., Li--Yau \cite{Li-Yau-1980}--\cite{Li-1982},
Zhong--Yang \cite{Zhong-Yang-1983}, Ling \cite{Jun Ling-2006-IJM}
(or \cite{JunLing-2007-AGAG}), and so on. One interesting feature of
our argument is that, it avoids various kinds of trouble from the singularity
of $\abs{\nabla u}^2/(1-u^2)$, which is already present in those references.
%(cf., \cite{Li-Yau-1980}, \cite{Li-1982} and \cite{Zhong-Yang-1983}, etc.).
%and so on.
%
%To overcome the difficult, our idea is to adopt $\cdots\cdots\cdots$
%
%We find a trick to overcome this difficulty.
Although with many ways analogous to Ling \cite{Jun Ling-2006-IJM}
(or \cite{JunLing-2007-AGAG}), we can readily
%handle
deal with the above singularity, and reduce the
difficulty in calculation to some degree.
%to a certain extent.
Maybe this is another way
%a new way
to estimate eigenvalues.

The remaining part of this paper is organized as follows.
%The paper is organized as follows.
%In the next section we prove a
In Section 2, we start by briefly introducing some terminologies
and notations, that is consistent with \cite{Schoen-Yau-1994},
where the corresponding term are defined in more general setting.
In Section 3, when $\p M\neq\emptyset$, firstly we establish a technique lemma,
which is recognized as another version of Lemma 2.2 in \cite{DaGang Yang-1999}.
With aid of this lemma and the maximum principle, we then establish
a rough estimate of $F(\theta)$ (its definition in \eqref{eqn-11-5-16-22-08} below).
%In Section 4 we get a more precise estimate on $F(\theta)$ than Lemma \ref{lem-11-5-21-21-16}.
A more precise estimate of $F(\theta)$ is provided at the end of
Section 4 via the method of barrier function.
It turns out that this improved estimate is essential in the
proof of Theorem \ref{thm-12-9-26-13a-15},
\ref{thm-Jun-Ling-2007-AGAG-2} and \ref{thm-12-9-27-8-6}.
%turns out to be essential in the proof of Theorem  \ref{thm-12-9-26-13a-15},
%\ref{thm-Jun-Ling-2007-AGAG-2} and  \ref{thm-12-9-27-8-6}.
Finally, as an application of the above estimate, the proof of all
these theorems mentioned are presented in Section 5. In Section 6,
we bring out an open problem in terms of the first eigenfunction.
May be this problem associated with eigenvalue estimate as well.
%We shall prove 111111111 results for 11111222 in Section 5,
%%and complete
%Section 3 provides a general setting in which assumptions of
%Section 2 are verified. Next, in Section 4 we show that

\section{Notations and preliminaries}
Let $\{e_1, e_2,\cdots, e_n\}$ be a local orthonormal frame field
on $M$. We adopt the notation that subscripts in $i, j$, and $k$,
with $1\ls i, j, k\ls n$, mean covariant differentiations in the
$e_i, e_j$ and $e_k$ directions respectively.

The Laplacian operator on $M$ in term of local coordinates
associated with the above orthonormal frame, is defined by
differentiating once more in the direction of $e_i$ and summing
over $i=1, 2,\cdots, n$, i. e.,
\[
\Delta u=\sum_i u_{ii}.
\]
Denote by $u$ the normalized eigenfunction with respect to the
first eigenvalue $-\lambda_1$ of $\Delta$. More precisely,
\begin{equation}\label{11-5-16-18-30}
\left\{
\begin{array}{l}
  \Delta u=-\lambda_1 u, \\
  \max\, u=1, \\
  \min\, u=-k,\quad 0<k\ls 1. \\
\end{array}
\right.
\end{equation}

Throughout this paper,
%Now if
we always set
\[
\theta(x)=\arcsin[u(x)],\quad \forall\,\,x\in M,
\]
and define a subset of
$M$ as follows
\[
\Sigma_*=\big\{x\in M\,\,:\,\,
\theta(x)=\frac{\pi}{2}\quad\hbox{or}\quad
\theta(x)=-\frac{\pi}{2}\quad\hbox{when}\quad k=1\big\}.
\]
Thus
\[
u(x)=\sin[\theta(x)],\quad\forall\,\,x\in M
\]
and
\[
-\arcsin k\ls \theta(x)\ls \frac{\pi}{2},\quad\forall\,\,x\in M.
\]
Above terms shall apply unless otherwise mention.

By \eqref{11-5-16-18-30}, a straight forward calculation
shows that $\theta(x)$ satisfies
%Then \eqref{11-5-16-18-35} transforms into
%Obviously, we obtain by \eqref{11-5-16-18-35} that $\theta:
%\bar{\Omega}\mapsto \mathbb{R}$ satisfies
\begin{equation}\label{11-5-18-10-58}
    \cos \theta \cdot \Delta \theta-\sin \theta\cdot \abs{\nabla
\theta}^2=-\lambda_1\sin\theta.
\end{equation}
In particular,
\begin{equation}\label{11-5-18-17-10}
\Delta \theta=\frac{\sin\theta}{\cos\theta}\cdot(\abs{\nabla\theta}^2-\lambda_1).
\end{equation}
whenever
%$\theta \in (-\frac{\pi}{2},\frac{\pi}{2})$ or
$x\in M\setminus \Sigma_*$. From \eqref{11-5-18-10-58}, we easily
know that
\begin{equation}\label{11-5-21-21-26}
\abs{\nabla\theta}^2=\lambda_1\qquad \hbox{as} \quad
\theta=\frac{\pi}{2},
\end{equation}
and when $k=1$,
\begin{equation}\label{11-5-21-21-30}
\abs{\nabla\theta}^2=\lambda_1\qquad \hbox{as} \quad
\theta=-\frac{\pi}{2}.
\end{equation}

%Also we
We also define a function $F$
%$F: [-\arcsin k,\frac{\pi}{2})\mapsto \mathbb{R}$
%$\big($ or $F: (-\frac{\pi}{2},\frac{\pi}{2})\mapsto \mathbb{R}$ when $k=1$$\big)$
as follows
\begin{equation}\label{eqn-11-5-16-22-08}
F(\theta_0)=\max_{x\in M,\,\theta(x)=\theta_0}\abs{\nabla\theta(x)}^2
\end{equation}
for all $\theta_0\in[-\arcsin k,\frac{\pi}{2})$
$\big($ or $(-\frac{\pi}{2},\frac{\pi}{2})$ when $k=1$$\big)$.
Obviously, $F$ is well-defined. Actually, $F(\theta_0)$ is not
something but an extreme value of $f$ with condition
$\theta(x)=\theta_0$. It is very easy to verify that $F(\theta)$
is continuous in $[-\arcsin k,\frac{\pi}{2})$
$\big($ or $(-\frac{\pi}{2},\frac{\pi}{2})$ when $k=1$$\big)$.
Moreover, by \eqref{11-5-21-21-26} and \eqref{11-5-21-21-30}, if we define
\[
F(\frac{\pi}{2})=F(\frac{\pi}{2}-0)=\lambda_1,
\]
and
\[
F(-\frac{\pi}{2})=F(-\frac{\pi}{2}+0)=\lambda_1\qquad\hbox{when}\quad k=1,
\]
then $F(\theta)$ can be extended a continuous function on
$[-\arcsin k,\frac{\pi}{2}]$ $\big($ or $[-\frac{\pi}{2},\frac{\pi}{2}]$ when $k=1$$\big)$.

\section{A rough estimate of $\abs{\nabla\theta}^2$}
Firstly in a similar way owing to \cite{Singer-Wong-Yau-Yau-1985},
\cite{Yu-Zhong-1986}, \cite{Schoen-Yau-1994}, also \cite{DaGang Yang-1999} and
\cite{Jun Ling-2006-IJM} (or \cite{JunLing-2007-AGAG}), we get the following lemma.
Actually, it can be viewed as another version of Lemma 2.2 in \cite{DaGang Yang-1999}.
\begin{lemma}\label{lem-12-9-25-20-26}
Suppose that $\p M\neq\emptyset$. Let $G(\theta)$ be a function defined as follows
\[
G(x)=\frac{1}{2}\abs{\nabla\theta(x)}^2+g[\theta(x)],\quad \forall\,\,x\in M,
\]
where $g(\theta)$ is a smooth function defined on $[-\arcsin k,\frac{\pi}{2}]$.
Then we have the following conclusions:

$(1)$\,\,Assume that the mean curvature $H$ of $\p M$ is nonnegative,
also $u$ satisfies the Dirichlet boundary condition,
%is the normalized Dirichlet eigenfunction,
and $g'(0)=0$.
If $G(x)$ arrives on its maximum at $x_0\in \p M\setminus \Sigma_*$,
then $\nabla G(x_0)=0$.
%Furthermore, we also have $\nabla\theta(x_0)=0$.

$(2)$\,\,Assume that the second fundamental form of $\p M$ is nonnegative
with respect to the outward normal $($i.e., weakly convex$)$,
also $u$ satisfies the Neumann boundary condition.
% is the normalized Neumann eigenfunction.
If $G(x)$ attains its maximum at $x_0\in \p M\setminus \Sigma_*$,
then $\nabla\theta(x_0)=0$. Furthermore,
%we also have
$\nabla G(x_0)=0$.
\end{lemma}
\begin{proof}
%We pick an orthonormal frame
Choose a local orthonormal frame $\{e_1, e_2,\cdots, e_n\}$ around %about
$x_0$ such that $e_1$ is the unit normal $\p M$ pointing
outward to $M$.
%the unit outward normal
%is perpendicular to $\p M$ and pointing outward.
%For the sake of simplicity,
We also denote below by $\frac{\p}{\p x_1}$ the restriction on $\p
M$ of the directional derivative corresponding to $e_1$.

\underline{The proof of (1)}:\,\,Clearly, the maximality of $G(x_0)$ implies that
\begin{equation}\label{eqn-12-9-25-22-15}
G_i(x_0)=0\quad\hbox{for}\quad2\ls i\ls n.
\end{equation}
%Clearly, we easily know that $\nabla_{e_1}e_i=0$ for $2\ls i\ls n$.
Since $u$ satisfies the Dirichlet boundary condition, then
\[
\theta|_{\p M}=(\arcsin u)|_{\p M}=0.
\]
Thus $\theta_i(x_0)=0$ for $2\ls i\ls n$.
We also derive from \eqref{11-5-18-17-10} that $(\Delta\theta)|_{\p M}=0$.
Using these results
%and some knowledge of geometry
in the following arguments, we have that at $x_0$
\begin{eqnarray*}
\frac{1}{2}\frac{\p(\abs{\nabla\theta}^2)}{\p x_1}
&=&\sum_{i=1}^n\theta_i\theta_{i1}=\theta_1\theta_{11}
=\theta_1(\Delta\theta-\sum_{i=2}^n\theta_{ii})\\
&=&-\theta_1\sum_{i=2}^n\theta_{ii}
=-\theta_1\sum_{i=2}^n(e_ie_i\theta-\nabla_{e_i}e_i\theta)\\
&=&\theta_1\sum_{i=2}^n\nabla_{e_i}e_i\theta
=\theta_1\sum_{i=2}^n\sum_{j=1}^n(\nabla_{e_i}e_i,e_j)\theta_j\\
&=&\theta_1^2\sum_{i=2}^n(\nabla_{e_i}e_i,e_1)
=-\theta_1^2\sum_{i=2}^n(\nabla_{e_i}e_1,e_i)\\
&=&-\theta_1^2\sum_{i=2}^nh_{ii}=-\theta_1^2H\ls0.
\end{eqnarray*}
Here $h_{ij}$ and $H=\sum_{i=2}^nh_{ii}$ are the second fundamental form
and the mean curvature of $\p M$ relative to $e_1$, respectively.
\begin{eqnarray*}
\frac{\p G}{\p x_1}(x_0)
&=&\frac{1}{2}\frac{\p(\abs{\nabla\theta}^2)}{\p x_1}(x_0)+g'[\theta(x_0)]\cdot\theta_1(x_0)\\
&\ls&g'[\theta(x_0)]\cdot\theta_1(x_0)=g'(0)\cdot\theta_1(x_0)=0.
\end{eqnarray*}
In addition, by the maximality of $G(x)$ at $x_0$, we also have
$\frac{\p G}{\p x_1}(x_0)\gs0$.
Thus
\begin{equation}\label{equ-12-9-25-22-29}
\frac{\p G}{\p x_1}(x_0)=0.
\end{equation}
Combining \eqref{eqn-12-9-25-22-15} with \eqref{equ-12-9-25-22-29}, we can get
\begin{equation*}\label{equ-12-9-25-22-33}
\nabla G(x_0)=0.
\end{equation*}

\vskip5pt

\underline{The proof of (2)}:\,\,By the maximality of $G(x_0)$, we also have
\begin{equation}\label{equ-12-9-27-21-56}
G_i(x_0)=0\qquad\hbox{for}\quad2\ls i\ls n
\end{equation}
and
\begin{equation}\label{11-6-12-1-29}
0\ls \frac{\p G}{\p x_1}(x_0)=\sum_{i=1}^n\theta_i(x_0)\cdot
\theta_{i1}(x_0)+g'[\theta(x_0)]\cdot \theta_1(x_0)
\end{equation}

In addition, since $u$ satisfies the Neumann boundary condition, then
\begin{eqnarray*}
\theta_1=\frac{1}{\sqrt{1-u^2}}\cdot u_1
=\frac{1}{\sqrt{1-u^2}}\cdot\frac{\p u}{\p x_1}=0
\qquad\hbox{on}\quad\p M.
\end{eqnarray*}
Therefore,
\begin{equation}\label{11-6-12-1-35}
\theta_1(x_0)=0.
\end{equation}
Putting \eqref{11-6-12-1-35} into \eqref{11-6-12-1-29}, we then
have
%Thus, using this fact we derive from \eqref{11-6-12-1-29} that
\begin{equation}\label{11-6-12-2-3}
0\ls \frac{\p G}{\p x_1}(x_0)=\sum_{i=2}^n\theta_i(x_0)\cdot
\theta_{i1}(x_0)
\end{equation}

Notice that $\theta_1(x_0)=0$ and recall the definition of second fundamental form
with respect to the outward normal,
one can derive that, for $2\ls i\ls n$
\begin{eqnarray*}\label{eqn-11-6-12-8-38}
\theta_{i1}&=&e_ie_1\theta-(\nabla_{e_i}e_1)\theta
=e_i(\theta_1)-(\nabla_{e_i}e_1,e_j)\theta_j\\
&=&-(\nabla_{e_i}e_1,e_j)\theta_j=-\sum_{j=2}^nh_{ij}\theta_j
\qquad\hbox{at}\,\,x_0,
\end{eqnarray*}
i.e., for $2\ls i\ls n$,
\begin{equation}\label{eqn-11-6-12-8-38}
\theta_{i1}=-\sum_{j=2}^nh_{ij}\theta_j
\qquad\hbox{at}\,\,x_0,
\end{equation}
where $(h_{ij})_{2\ls i, j\ls n}$ is the second fundamental form of $\p M$ relative to $e_1$.
Putting \eqref{eqn-11-6-12-8-38} into \eqref{11-6-12-2-3}, we can get
\begin{equation}\label{equ-12-9-27-22-1}
0\ls \frac{\p G}{\p x_1}(x_0)=-\sum_{i,j=2}^n\theta_i(x_0)h_{ij}(x_0)\theta_j(x_0)\ls 0,
\end{equation}
since $(h_{ij})_{2\ls i, j\ls n}$ is nonnegative (i.e., $\p M$ is weakly convex).
Hence, $\theta_i(x_0)=0$ for $2\ls i\ls n$. By \eqref{11-6-12-1-35}
%the fact $\theta_1(x_0)=0$
again, we have $\nabla\theta(x_0)=0$. Finally, $\nabla G(x_0)=0$ follows from
\eqref{equ-12-9-27-21-56} and \eqref{equ-12-9-27-22-1}.

%So we conclude the proof of this lemma.
So far we finish the proof of this lemma.
\end{proof}

%As
%Just as \cite{Zhong-Yang-1983} points out that the estimate of
%the upper bound of $\abs{\nabla\theta}^2$ plays an important role
%in the estimate of the lower bound for $\lambda_1$.
It was just as Zhong--Yang \cite{Zhong-Yang-1983} had pointed out
that the estimate of the upper bound of $\abs{\nabla\theta}^2$
plays an important role in the estimate of the lower bound for
$\lambda_1$.
%Just as \cite{Zhong-Yang-1983} points out that the estimate of the upper
%bound of $\abs{\nabla\theta}^2$ plays an important role in the
%estimate of the lower bound for $\lambda_1$.
In the following we
establish a rough estimate for $\abs{\nabla\theta}^2$.
\begin{lemma}\label{lem-11-5-21-21-16}
%$($see \cite{Zhong-Yang-1983}$)$
Assume that $\mathrm{Ric}(M)\gs0$. The other assumption as in Theorem \ref{thm-Lichnerowicz}.
%Let $\lambda_1$ be the first positive eigenvalue on
%$M$ $($with either Dirichlet or Neumann boundary condition if $\p
%M\neq\emptyset$$)$. If $\p M\neq\emptyset$, we also assume that $\p
%M$ is of nonnegative mean curvature
%%$\mathrm{tr} S\gs 0$
%if $\lambda_1$ is the first Dirichlet eigenvalue and $\p M$ is of
%nonnegative definite second fundamental form
%%$S\gs 0$
%if $\lambda_1$ is the first Neumann eigenvalue.
In any case, the following estimate is valid.
\begin{equation}\label{11-5-16-13-30}
\abs{\nabla\theta(x)}^2\ls \lambda_1,\quad \forall\,\,x\in M.
\end{equation}
Moreover,
\begin{equation}\label{11-5-25-20-26}
F(\theta)\ls \lambda_1.
\end{equation}
\end{lemma}
\begin{proof}
Suppose that $\abs{\nabla\theta}^2$ attains its local maximum
%(or extreme value)
at $x_0$. Clearly, \eqref{11-5-21-21-26} and \eqref{11-5-21-21-30}
imply that \eqref{11-5-16-13-30} holds in the case: $x_0\in
\Sigma_*$. Without loss of generality, we may assume further that
$x_0\in M\setminus\Sigma_*$ in the rest of the proof, thus
$\theta_0=\theta(x_0) \in [-\arcsin k,\frac{\pi}{2})$ (or
$(-\frac{\pi}{2},\frac{\pi}{2})$ when $k=1$). In the case of $\p M\neq\emptyset$,
with aid of Lemma \ref{lem-12-9-25-20-26},
%and the maximum principle,
we conclude that if $\abs{\nabla\theta}^2$ arrive its maximum at $x_0\in M$,
then
\begin{equation}\label{eqn-12-9-26-20-3}
\nabla(\abs{\nabla\theta}^2)=0\qquad\hbox{at}\quad x_0.
\end{equation}
no matter $x_0\in\p M\setminus\Sigma_*$ or $x_0\in M\setminus(\p M\cup\Sigma_*)$.
According to %the  knowledge of advanced calculus,
the maximum principle again, we easily show that
\begin{equation}\label{11-5-18-16-38}
 \Delta (\abs{\nabla\theta}^2)\ls 0\qquad\hbox{at}\quad x_0.
\end{equation}
Applying the Bochner formula to $\theta$, we have
\begin{equation}\label{11-5-18-12-00}
  \frac{1}{2}\,\Delta (\abs{\nabla
\theta}^2)=
  \abs{\nabla^2\theta}^2+\nabla\theta\cdot \nabla (\Delta
  \theta)+\textrm{Ric}(\nabla\theta,\nabla\theta),
\end{equation}
%where $R_{ljij}$ denote the Riemannian curvature tensor,
%$R_{li}=\sum_j R_{ljij}$ are the Ricci curvature components,
%whereas
%\[
%\textrm{Ric}(\nabla\theta,\nabla\theta)=\sum_l
%R_{li}\theta_l\theta_i
%\]
where $\textrm{Ric}(\nabla\theta,\nabla\theta)$ is the Ricci
curvature along %the gradient
%with respect to
$\nabla\theta$. Substituting \eqref{11-5-18-17-10} into
\eqref{11-5-18-12-00}, we have
\begin{eqnarray}\label{11-6-14-3-1}
\frac{1}{2}\,\Delta (\abs{\nabla\theta}^2) &=&
\abs{\nabla^2\theta}^2+\nabla\theta\cdot \nabla \big[\frac{\sin
\theta}{\cos \theta}\cdot (\abs{\nabla\theta}^2-\lambda_1)\big]
+\mathrm{Ric}(\nabla\theta,\nabla\theta) \nonumber\\
&=& \abs{\nabla^2\theta}^2+\nabla\theta\cdot
\nabla \big(\frac{\sin \theta}{\cos \theta}\big)
\cdot(\abs{\nabla\theta}^2-\lambda_1) \\
&&+\nabla\theta\cdot \frac{\sin \theta}{\cos \theta}\cdot
\nabla(\abs{\nabla\theta}^2)+\mathrm{Ric}(\nabla\theta,\nabla\theta).
\nonumber
\end{eqnarray}
A direct calculation leads to that
\begin{equation}\label{11-5-18-16-20}
    \nabla \big(\frac{\sin \theta}{\cos \theta}\big)
=\frac{\nabla(\sin \theta)\cdot \cos \theta-\sin \theta \cdot
\nabla (\cos \theta)}{\cos^2\theta}=\frac{1}{\cos^2\theta}\cdot\nabla\theta,
\end{equation}
%By virtue of \eqref{11-6-14-3-1}--\eqref{11-5-18-16-25},
Putting \eqref{11-5-18-16-20} into \eqref{11-6-14-3-1}, we obtain
\begin{eqnarray}\label{11-5-18-16-30}
   \frac{1}{2}\,\Delta (\abs{\nabla\theta}^2) &=& \abs{\nabla^2\theta}^2
   +\frac{1}{\cos^2\theta}\cdot\abs{\nabla\theta}^2(\abs{\nabla\theta}^2-\lambda_1) \nonumber\\
   &&+\nabla\theta\cdot \frac{\sin \theta}{\cos \theta}\cdot \nabla (\abs{\nabla
   \theta}^2)+\mathrm{Ric}(\nabla\theta,\nabla\theta).
\end{eqnarray}
%Observe that
By virtue of \eqref{eqn-12-9-26-20-3}--\eqref{11-5-18-16-38}, we deduce from \eqref{11-5-18-16-30} that
at $x_0$
\[
0\gs\abs{\nabla^2\theta}^2+\frac{1}{\cos^2\theta}
\cdot\abs{\nabla\theta}^2(\abs{\nabla\theta}^2-\lambda_1)
+\mathrm{Ric}(\nabla\theta,\nabla\theta).
\]
%Under the assumption that
Since $\mathrm{Ric}(M)\gs0$, the first term and the third term above
%on the right-hand side
can be taken away
%from the above inequality
since they are nonnegative. Thus we obtain at $x_0$
\[
0\gs \frac{1}{\cos^2\theta}\cdot\abs{\nabla\theta}^2(\abs{\nabla\theta}^2-\lambda_1).
\]
%Thus at $x_0$
Dividing by $\abs{\nabla\theta}^2$ and multiplying by
$\cos^2\theta$ successively, it follows that at $x_0$
\[
0\gs \abs{\nabla\theta}^2-\lambda_1.
\]
Hence we have
\[
\abs{\nabla\theta(x_0)}^2\ls \lambda_1,
\]
%where $\theta_0=\theta(x_0)$.
%This completes the proof.
%which proves the lemma.
which implies the conclusion.
\end{proof}

\section{The estimate of $F(\theta)$}
%From Li-Yau's proof we that, if $k=1$, i.e.,
%$a=\frac{1-k}{1+k}=0$, it is easy to get the desired result
%$\lambda_1\gs \frac{\pi^2}{d^2}$.
%In the sequel, what we want now is get a more precise estimate on
Now we are trying to
%are going to
get a more precise estimate on $F(\theta)$ than
Lemma \ref{lem-11-5-21-21-16}. For this purpose, let us introduce the function
$Z(\theta): [-\arcsin k,\frac{\pi}{2}]\mapsto\mathbb{R}$ such that
\begin{equation}\label{11-5-25-22-01}
F(\theta)=\lambda_1Z(\theta).
\end{equation}
By Lemma \ref{lem-11-5-21-21-16}, it is also easy to see that $0\ls Z(\theta)\ls1$.
%Throughout this paper,
From now on we denote
\[
\delta=\frac{(n-1)K}{2\lambda_1}.
\]
It follows from \eqref{Lichnerowicz-estimate} that
\[
0<\delta\ls\frac{n-1}{2n}<\frac{1}{2}.
\]

\begin{lemma}\label{lem-11-5-21-21-50}
Assume that $\mathrm{Ric}(M)\gs(n-1)K$ and the other conditions as in Theorem \ref{thm-Lichnerowicz}.
If the function $z: [-arc\sin k,\frac{\pi}{2}]\mapsto \mathbb{R}$ satisfies the following properties:

(1)\,\,$z(\theta)\gs Z(\theta)$;

(2)\,\,there exists some $\theta_0\in [-arc\sin k,\frac{\pi}{2})$
$\big($or $(-\frac{\pi}{2},\frac{\pi}{2})$ when $k=1$$\big)$, such that $z(\theta_0)=Z(\theta_0)$;

%(3)\,\,$z$ extends to a smooth even function; and
(3)\,\,$z'(0)=0$;

(4)\,\,$z'(\theta_0)\sin\theta_0>0$.

Then the following estimate holds
\begin{equation}\label{11-5-21-22-10}
z(\theta_0)\ls1-\cos\theta_0\sin \theta_0\cdot z'(\theta_0)
+\frac{\cos^2\theta_0}{2}\cdot
z''(\theta_0)-2\delta\cos^2\theta_0.
\end{equation}
\end{lemma}
\begin{proof}
Set
\[
f(x)=\frac{1}{2}\,\Big\{\abs{\nabla\theta(x)}^2-\lambda_1z\big[\theta(x)\big]\Big\}.
\]
Obviously, $f(x)\ls 0$ for all $x\in M$.
By \eqref{eqn-11-5-16-22-08}, we know that there exists some $x_0\in
M\setminus\Sigma_*$ such that $\theta(x_0)=\theta_0$ and
$F(\theta_0)=\abs{\nabla\theta(x_0)}^2$.
Thus $f$ achieves its maximum $0$
%(or extreme value)
at $x_0$, i. e.,
\begin{equation}\label{11-5-22-9-33}
\abs{\nabla\theta(x_0)}^2=\lambda_1Z(\theta_0)=\lambda_1z(\theta_0).
\end{equation}
By the same reason as in the proof of lemma \ref{lem-11-5-21-21-16},
we always have
\begin{equation}\label{eqn-12-9-26-20-18}
\nabla f(x_0)=0,
\end{equation}
no matter $x_0\in\p M\setminus\Sigma_*$ or $x_0\in M\setminus(\p M\cup\Sigma_*)$.
It is obvious that
\begin{equation}\label{11-5-21-22-35}
\Delta f(x_0)\ls0.
\end{equation}
%A direct calculation yields
by the maximum principle again. Direct computation shows that
\[
f_j=\sum_i\theta_i\cdot \theta_{ij}-\frac{\lambda_1}{2}z'(\theta)\cdot \theta_j,
\]
%or equivalently,
%namely,
that is
\begin{equation*}\label{11-6-19-14-30}
\nabla f=\frac{1}{2}\,[\nabla (\abs{\nabla\theta}^2)-\lambda_1
z'(\theta)\cdot \nabla\theta] =\nabla\theta\cdot \nabla^2
\theta-\frac{\lambda_1}{2}z'(\theta)\cdot\nabla\theta.
\end{equation*}

%So, we have at $x_0$
Since $\nabla f=0$ at $x_0$,
%\begin{equation}\label{11-5-22-8-30}
%\theta_i\cdot \theta_{ij}=\frac{\lambda_1}{2}Z'(\theta_0)\cdot
%\theta_j,
%\end{equation}
%%or equivalently,
%namely,
\begin{equation}\label{11-5-22-8-35}
\nabla (\abs{\nabla\theta}^2)=2\nabla\theta\cdot \nabla^2
\theta=\lambda_1z'(\theta_0)\cdot \nabla\theta\qquad
\hbox{at}\quad x_0.
\end{equation}

By directly calculating and applying \eqref{11-5-18-17-10}, we
also obtain
\begin{eqnarray}\label{11-5-21-23-58}
  \frac{\lambda_1}{2}\,\Delta z
&=& \frac{\lambda_1}{2}\,\sum_jz_{jj}=\frac{\lambda_1}{2} \sum_j
\big(z'\cdot \theta_j\big)_j \nonumber\\
  &=& \frac{\lambda_1}{2} \sum_j (z''\cdot \theta_j^2+z'\cdot \theta_{jj})
  =\frac{\lambda_1}{2}(z''\cdot \abs{\nabla\theta}^2+z'\cdot \Delta
  \theta) \\
  &=& \frac{\lambda_1}{2}\Big[z''\cdot \abs{\nabla\theta}^2
+z'\cdot\frac{\sin\theta}{\cos \theta}\cdot(\abs{\nabla\theta}^2-\lambda_1)\Big]. \nonumber
\end{eqnarray}
Combining \eqref{11-5-18-16-30} with \eqref{11-5-21-23-58}, we
hence obtain
\begin{eqnarray*}
\Delta f&=&\abs{\nabla^2\theta}^2
+\frac{1}{\cos^2\theta}\cdot\abs{\nabla\theta}^2(\abs{\nabla\theta}^2-\lambda_1) \nonumber\\
&&+\nabla\theta\cdot \frac{\sin \theta}{\cos \theta}\cdot \nabla (\abs{\nabla
\theta}^2)+\mathrm{Ric}(\nabla\theta,\nabla\theta)\\
&&-\frac{\lambda_1}{2}\Big[z''\cdot \abs{\nabla\theta}^2
+z'\cdot\frac{\sin\theta}{\cos \theta}\cdot(\abs{\nabla\theta}^2-\lambda_1)\Big]. \nonumber
\end{eqnarray*}
Recall that $\mathrm{Ric}(\nabla\theta,\nabla\theta)\gs (n-1)K\abs{\nabla\theta}^2$, we can get
\begin{eqnarray}\label{12-6-18-20-55}
\Delta f &=& \abs{\nabla^2\theta}^2
+\frac{1}{\cos^2\theta}\cdot\abs{\nabla\theta}^2(\abs{\nabla\theta}^2-\lambda_1) \nonumber\\
&&+\nabla\theta\cdot \frac{\sin \theta}{\cos \theta}\cdot \nabla (\abs{\nabla
\theta}^2)+(n-1)K\abs{\nabla\theta}^2 \\
&&-\frac{\lambda_1}{2}\Big[z''\cdot \abs{\nabla\theta}^2
+z'\cdot\frac{\sin\theta}{\cos \theta}\cdot(\abs{\nabla\theta}^2-\lambda_1)\Big]. \nonumber
\end{eqnarray}
%Plugging
%Putting
Substituting \eqref{11-5-22-8-35} into \eqref{12-6-18-20-55}, it is
easy to deduce that at $x_0$
\begin{eqnarray}\label{11-5-21-9-22}
   \Delta f &=& \abs{\nabla^2\theta}^2
   +\frac{1}{\cos^2\theta}\cdot\abs{\nabla\theta}^2(\abs{\nabla\theta}^2-\lambda_1)\nonumber\\
   &&+\lambda_1z'\cdot \frac{\sin \theta}{\cos \theta}\cdot \abs{\nabla
   \theta}^2+(n-1)K\abs{\nabla\theta}^2\\
   &&-\frac{\lambda_1}{2}\Big[z''\cdot \abs{\nabla\theta}^2
   +z'\cdot\frac{\sin\theta}{\cos \theta}\cdot(\abs{\nabla\theta}^2-\lambda_1)\Big].\nonumber
\end{eqnarray}
By virtue of \eqref{11-5-21-22-35}, we
derive from \eqref{11-5-21-9-22} that at $x_0$
\begin{eqnarray}\label{11-5-21-10-25}
   0 &\gs& \abs{\nabla^2\theta}^2
   +\frac{1}{\cos^2\theta}\cdot\abs{\nabla\theta}^2(\abs{\nabla\theta}^2-\lambda_1)\nonumber\\
   &&+\lambda_1z'\cdot \frac{\sin \theta}{\cos \theta}\cdot \abs{\nabla
   \theta}^2+(n-1)K\abs{\nabla\theta}^2\\
   &&-\frac{\lambda_1}{2}z''\cdot \abs{\nabla\theta}^2
   +\frac{\lambda_1}{2}\cdot\frac{z'\sin\theta}{\cos\theta}
   \cdot(\lambda_1-\abs{\nabla\theta}^2).\nonumber
\end{eqnarray}
%Eliminating \eqref{11-5-21-10-25} by $\lambda_1^2 a$ and multiplying
%by $\cos^2 \theta$ successively, it follows
Obviously, condition (4) in this theorem and \eqref{11-5-16-13-30} imply that
the last term in \eqref{11-5-21-10-25} is nonnegative.
Thus the first term and the last term above
%on the right-hand side
can be discarded
%dropped
since they are nonnegative. We thus
%hence
obtain that at $x_0$
\begin{eqnarray*}
0&\gs&\frac{1}{\cos^2\theta}\cdot\abs{\nabla\theta}^2(\abs{\nabla\theta}^2-\lambda_1)
+\lambda_1z'\cdot\frac{\sin \theta}{\cos \theta}\cdot\abs{\nabla\theta}^2\nonumber\\
&&+(n-1)K\abs{\nabla\theta}^2-\frac{\lambda_1}{2}z''\cdot \abs{\nabla\theta}^2.\nonumber
\end{eqnarray*}
After dividing by $\lambda_1\abs{\nabla\theta}^2$, multiplying by $\cos^2\theta$
and rearranging the terms successively, we are led to at $x_0$
\begin{eqnarray}\label{11-5-22-12-01}
0&\gs&\frac{\abs{\nabla\theta}^2}{\lambda_1}-1
+z'\cdot\cos\theta\sin\theta
-\frac{1}{2}z''\cos^2\theta+2\delta\cos^2\theta.
\end{eqnarray}
Therefore, using \eqref{11-5-22-9-33} we get at $x_0$
\begin{eqnarray}\label{eqn-12-9-23a-19-55}
0&\gs&z-1+z'\cdot\cos\theta\sin\theta
-\frac{1}{2}z''\cos^2\theta+2\delta\cos^2\theta.
\end{eqnarray}
%Obviously, \eqref{11-5-21-22-10} follows from \eqref{eqn-12-9-23a-19-55} immediately.
from which \eqref{11-5-21-22-10} follows easily. The proof is complete.
%which proves the lemma.
%Therefore, the conclusion is valid.
\end{proof}

%We also
We would like to point out that the remaining part of the
present paper works exactly as in \cite{Zhong-Yang-1983} (or
\cite{Schoen-Yau-1994}) and \cite{Jun Ling-2006-IJM} (or
\cite{JunLing-2007-AGAG}). For the completeness we briefly give a
proof of Theorem \ref{thm-12-9-26-13a-15}, \ref{thm-Jun-Ling-2007-AGAG-2}
and \ref{thm-12-9-27-8-6} below which only use the methods due to
\cite{Zhong-Yang-1983} and \cite{Jun Ling-2006-IJM}
(or \cite{JunLing-2007-AGAG}). We refer the interested reader to
consult these references for more details.
\begin{lemma}\label{lem-12-6-19-7-50}
$($see \cite{Jun Ling-2006-IJM}, or \cite{JunLing-2007-AGAG}$)$ Let
\begin{equation}\label{eqn-12-6-19-8-0}
\xi(\theta)=\frac{\cos^2\theta+2\theta\sin\theta\cos\theta+\theta^2-\frac{\pi^2}{4}}{\cos^2\theta}
\quad\hbox{in}\quad (-\frac{\pi}{2},\frac{\pi}{2})
\end{equation}
and $\xi(\pm\frac{\pi}{2})=0$.
Then the function $\xi$ satisfies the following
\begin{equation}\label{12-6-19-8-8}
\frac{\cos^2\theta}{2}\cdot\xi''-\cos\theta\sin\theta\cdot\xi'-\xi=2\cos^2\theta
\quad\hbox{in}\quad (-\frac{\pi}{2},\frac{\pi}{2}),
\end{equation}
Moreover, the function $\xi$ also has the following properties:
\[
\xi(-\theta)=\xi(\theta),\quad\forall\,\,\theta\in(-\frac{\pi}{2},\frac{\pi}{2});
\]
\[
\int_{0}^{\frac{\pi}{2}}\xi(\theta)\mathrm{d}\theta=-\frac{\pi}{2};
\]
\[
\xi'(\theta)<0\quad\hbox{on}\quad (-\frac{\pi}{2},0)
\qquad\hbox{and}\qquad \xi'(\theta)>0\quad\hbox{on}\quad (0,\frac{\pi}{2});
\]
\begin{equation}\label{inequ-12-9-27-11-12}
1-\frac{\pi^2}{4}=\xi(0)\ls\xi(\theta)\ls\xi(\pm\frac{\pi}{2})=0
\quad\hbox{on}\quad [-\frac{\pi}{2},\frac{\pi}{2}];
\end{equation}
%\[
%\xi'\,\,\textrm{is increasing on}\,\,[-\frac{\pi}{2},\frac{\pi}{2}]
%\qquad\hbox{and}\qquad\xi'(\pm\frac{\pi}{2})=\pm\frac{2\pi}{3};
%\]
%\[
%\Big(\frac{\xi'(\theta)}{\theta}\Big)'>0\quad\hbox{on}\quad (0,\frac{\pi}{2})
%\qquad\hbox{and}\qquad 2(3-\frac{\pi^2}{4})
%\ls\frac{\xi'(\theta)}{\theta}\ls\frac{4}{3}
%\quad\hbox{on}\quad [-\frac{\pi}{2},\frac{\pi}{2}].
%\]
\end{lemma}
\begin{corollary}\label{cor-12-6-19-11-20}
Let
\begin{equation}\label{eqn-12-6-19-11-30}
z(\theta)=1+\delta\,\xi(\theta).
\end{equation}
Then $z$ satisfies the
following
\begin{equation}\label{eqn-12-7-8-11-18-22}
\frac{\cos^2\theta}{2}\cdot z''(\theta)
-\cos\theta\sin\theta\cdot
z'(\theta)-z(\theta)+1=2\delta\cos^2\theta\quad\hbox{in}\quad
(-\frac{\pi}{2},\frac{\pi}{2}).
\end{equation}
%Moreover,
\[
z(\theta)>0,\quad\forall\,\,\theta \in [-\frac{\pi}{2},\frac{\pi}{2}];
\]
\[
z'(0)=\delta\xi'(0);
\]
\[
z'(\theta)\sin\theta=\delta\,\xi'(\theta)\sin\theta\gs
0,\quad\forall\,\,\theta \in [-\frac{\pi}{2},\frac{\pi}{2}].
\]
\end{corollary}
\begin{proof}
Using \eqref{inequ-12-9-27-11-12}, we easily get
\begin{eqnarray*}
z(\theta)&=&1+\delta\xi(\theta)\gs1+\delta\xi(0)\gs1+\delta(1-\frac{\pi^2}{4})\\
&>&1+\frac{1}{2}(1-\frac{\pi^2}{4})=\frac{3}{2}-\frac{\pi^2}{8}\approx0.26>0.
\end{eqnarray*}
In addition, by Lemma \ref{lem-12-6-19-7-50}, it is direct to verify
%check
the other properties. %The proof is complete.
\end{proof}

Using Lemma \ref{lem-11-5-21-21-50} and Corollary \ref{cor-12-6-19-11-20} and
the reduction to absurdity,
%the method of proof by contradiction
we easily prove the following conclusion. For the reader's convenience,
we give a proof below which is first due to \cite{Zhong-Yang-1983},
also due to \cite{Jun Ling-2006-IJM} (or \cite{JunLing-2007-AGAG}).
\begin{lemma}\label{11-5-25-22-05}
Assume that $Z(\theta)$ and $z(\theta)$ are defined by \eqref{11-5-25-22-01}
and \eqref{eqn-12-6-19-11-30}, respectively. Then
\begin{equation}\label{12-5-1-5-26}
Z(\theta)\ls z(\theta).
\end{equation}
\end{lemma}
\begin{proof}
Assume that \eqref{12-5-1-5-26} is not true. Since
$Z(\frac{\pi}{2})=1=z(\frac{\pi}{2})$, then there exists some
$\theta_0\in[-\arcsin k,\frac{\pi}{2})$
$\big($or $(-\frac{\pi}{2},\frac{\pi}{2})$ when $k=1$$\big)$ such that
\begin{equation}\label{12-5-1-6-10}
\sigma=Z(\theta_0)-z(\theta_0)=\max_{0\ls\theta\ls\frac{\pi}{2}}
\{Z(\theta)-z(\theta)\}>0.
\end{equation}
Set $\tilde{z}(\theta_0)=z(\theta)+\sigma$. Obviously,
\[
\tilde{z}(\theta)=z(\theta)+\sigma\gs
z(\theta)+[Z(\theta)-z(\theta)]=Z(\theta),
\]
\[
\tilde{z}(\theta_0)=z(\theta_0)+\sigma=Z(\theta_0),
\]
\[
\tilde{z}'(\theta)=z'(\theta),
\]
\[
\tilde{z}'(\theta_0)\sin\theta_0=z'(\theta_0)\sin\theta_0\gs 0.
\]

In place of $z(\theta)$ in Lemma \ref{lem-11-5-21-21-50} by
$\tilde{z}(\theta)$, we deduce by Lemma
\ref{lem-11-5-21-21-50} and \eqref{eqn-12-7-8-11-18-22} that
\begin{eqnarray*}
Z(\theta_0)&=&\tilde{z}(\theta_0)
\ls1-\cos\theta_0\sin\theta_0\cdot\tilde{z}'(\theta_0)
+\frac{\cos^2\theta_0}{2}\cdot\tilde{z}''(\theta_0)-2\delta\cos^2\theta_0\\
&=&1-\cos\theta_0\sin\theta_0\cdot z'(\theta_0)
+\frac{\cos^2\theta_0}{2}\cdot
z''(\theta_0)-2\delta\cos^2\theta_0=z(\theta_0).
\end{eqnarray*}
But this contradicts \eqref{12-5-1-6-10}. The proof is complete.
\end{proof}

\begin{corollary}\label{11-5-25-22-15}
%Under the assumption in
The assumption as in Theorem \ref{thm-Lichnerowicz}.
%The first nonzero (closed, or Dirichlet, or Neumann) eigenvalue $\lambda_1$ has the following estimate
In any case, the following estimate holds.
\begin{equation}\label{11-5-25-22-26}
F(\theta)\ls \lambda_1z(\theta),
\end{equation}
where $F(\theta)$ and $z(\theta)$ are defined by
\eqref{eqn-11-5-16-22-08} and \eqref{eqn-12-6-19-11-30},
respectively.
\end{corollary}

Our argument above establishes the inequality
\eqref{11-5-25-22-26}, which is an improved estimate of the upper
bound for $F(\theta)$ as required.

\section{Proof of some theorem}
%We stress here that although a reasoning similar to the one in
%\cite{Zhong-Yang-1983} and \cite{Yu-Zhong-1986} will give the
%claim, our proof is slightly different from
%\cite{Zhong-Yang-1983}, \cite{Yu-Zhong-1986} and
%\cite{Schoen-Yau-1994}.
%slightly differently
%Following a slightly different argument in \cite{Zhong-Yang-1983}
Following the same argument in \cite{Jun Ling-2006-IJM} (or \cite{JunLing-2007-AGAG}),
we now use the estimate of $F(\theta)$ to prove Theorem \ref{thm-12-9-26-13a-15} as follows.
%Applying the estimate for $F(\theta)$, we can give the following
%proof of Theorem \ref{11-5-23-21-02}.
\begin{proof}
\eqref{11-5-25-22-26} implies that
\[
\sqrt{\lambda_1}\gs\sqrt{\frac{\abs{F(\theta)}}{z(\theta)}}
=\sqrt{\frac{\abs{F(\theta)}}{1+\delta\,\xi(\theta)}}\gs
\frac{\abs{\nabla\theta}}{\sqrt{1+\delta\,\xi(\theta)}},
\]
i.e.,
\begin{equation}\label{11-5-25-22-31}
\sqrt{\lambda_1}\gs\frac{\abs{\nabla\theta}}{\sqrt{1+\delta\,\xi(\theta)}},
\end{equation}
where $\xi(\theta)$ is defined by \eqref{eqn-12-6-19-8-0}.

Take $q_1\in M$ such that $\theta(q_1)=\frac{\pi}{2}$.
Choose $q_2\in\p M$ such $\dist(q_1,q_2)=\dist(q_1,\p M)$. Clearly $\theta(q_2)=0$.
We denote by $d'$ the length of a
shortest curve $\gamma$ which connects $q_1$ with $q_2$ on $M$.
%We denote by $d'$ the length of the shortest curve $\gamma$ on $M$
%joining $x_1$ and $x_2$.
Let $\tilde{d}$ be the diameter of the largest interior ball in $M$
(i.e., $\tilde{d}/2$ is the inscribed radius of $M$).
%(i.e., $\tilde{d}/2$ is the inscribed radius for $M$).
Clearly, $d'\ls\tilde{d}/2$.
Integrating both sides of
\eqref{11-5-25-22-31} along the curve $\gamma$, we derive the following
\begin{eqnarray*}
\sqrt{\lambda_1}\,\frac{\tilde{d}}{2} &\gs& \sqrt{\lambda_1}\,d'
=\int_\gamma\sqrt{\lambda_1}\mathrm{d}s
\gs\int_\gamma\frac{1}{\sqrt{1+\delta\,\xi(\theta)}}\abs{\nabla\theta} \mathrm{d}s\\
&\gs&\int_\gamma\frac{1}{\sqrt{1+\delta\,\xi(\theta)}}
\mathrm{d}\theta =\int_{0}^\frac{\pi}{2}
\frac{1}{\sqrt{1+\delta\,\xi(\theta)}}
\mathrm{d}\theta\\[10pt]
&\gs&\Big(\int_{0}^\frac{\pi}{2}\mathrm{d}\theta\Big)^\frac{3}{2}\Big/
\Big\{\int_{0}^\frac{\pi}{2}[1+\delta\,\xi(\theta)]\mathrm{d}\theta\Big\}^\frac{1}{2}
\\[10pt]
&=&(\frac{\pi}{2})^\frac{3}{2}\Big/
\Big\{\int_{0}^\frac{\pi}{2}[1+\delta\,\xi(\theta)]\mathrm{d}\theta\Big\}^\frac{1}{2}.
\end{eqnarray*}
So, dividing by $\frac{\tilde{d}}{2}$ and squaring the two sides successively, we have
\[
\lambda_1\gs\frac{\pi^3}{2\tilde{d}^2}\Big/
\int_{0}^\frac{\pi}{2}[1+\delta\,\xi(\theta)]\mathrm{d}\theta.
\]
On the other hand,
\[
\int_{0}^\frac{\pi}{2}[1+\delta\,\xi(\theta)]\mathrm{d}\theta
=\frac{\pi}{2}+\delta\int_{0}^\frac{\pi}{2}\xi(\theta)\mathrm{d}\theta
=\frac{\pi}{2}(1-\delta).
\]
Hence, we conclude that
\[
\lambda_1\gs\frac{1}{1-\delta}\cdot\frac{\pi^2}{\tilde{d}^2},
\]
or, equivalently:
%This implies
\[
\lambda_1(1-\delta)\gs\frac{\pi^2}{\tilde{d}^2}.
\]
Therefore, we obtain
\[
\lambda_1\gs\frac{\pi^2}{\tilde{d}^2}+\lambda_1\delta
=\frac{\pi^2}{\tilde{d}^2}+\frac{(n-1)K}{2}.
\]
This completes the proof.
%Now we obtain the required estimate.
%Now we obtain the sharp estimate $\lambda_1\gs \pi^2/d^2$.
%This concludes the proof of the theorem.
%This implies \eqref{11-5-23-21-02}. So far we complete the proof
%of Theorem \ref{Zhong-Yang-1983 Main Theorem}.
\end{proof}

Using the similar argument as in the proof of Theorem
\ref{thm-12-9-26-13a-15}, we prove Theorem
\ref{thm-Jun-Ling-2007-AGAG-2} and Theorem \ref{thm-12-9-27-8-6}
together in the following.
\begin{proof}
Take $x_1, x_2\in M$ such that $\theta(x_1)=-\frac{\pi}{2}$,
$\theta(x_2)=\frac{\pi}{2}$. We denote by $d'$ the length of a
shortest curve $\gamma$ which connects $x_1$ with $x_2$ on $M$.
%We denote by $d'$ the length of the shortest curve $\gamma$ on $M$
%joining $x_1$ and $x_2$.
Let $d$ be the diameter of $M$.
Clearly, $d'\ls d$. Integrating both sides of \eqref{11-5-25-22-31}
along the curve $\gamma$, we derive the following
\begin{eqnarray*}
\sqrt{\lambda_1}d &\gs& \sqrt{\lambda_1}d'
=\int_\gamma \sqrt{\lambda_1}\mathrm{d}s
\gs\int_\gamma \frac{1}{\sqrt{1+z(\theta)}}\abs{\nabla\theta} \mathrm{d}s\\
&\gs&\int_\gamma \frac{1}{\sqrt{1+z(\theta)}} \mathrm{d}\theta
=\int_{-\frac{\pi}{2}}^\frac{\pi}{2} \frac{1}{\sqrt{1+z(\theta)}}
\mathrm{d}\theta\\[10pt]
&\gs&\Big(\int_{-\frac{\pi}{2}}^\frac{\pi}{2}\mathrm{d}\theta\Big)^\frac{3}{2}\Big/
\Big\{\int_{-\frac{\pi}{2}}^\frac{\pi}{2}[1+z(\theta)]\mathrm{d}\theta\Big\}^\frac{1}{2}
\\[10pt]
&=&\pi^\frac{3}{2}\Big/
\Big\{\int_{-\frac{\pi}{2}}^\frac{\pi}{2}[1+z(\theta)]\mathrm{d}\theta\Big\}^\frac{1}{2}.
\end{eqnarray*}
So we have
\[
\lambda_1\gs\frac{\pi^3}{d^2}\Big/
\int_{-\frac{\pi}{2}}^\frac{\pi}{2}[1+z(\theta)]\mathrm{d}\theta.
\]
In addition,
\[
\int_{-\frac{\pi}{2}}^\frac{\pi}{2}[1+z(\theta)]\mathrm{d}\theta
=\pi(1-\theta).
\]
Thus
\[
\lambda_1\gs\frac{1}{1-\delta}\cdot\frac{\pi^2}{d^2},
\]
or, equivalently:
\[
\lambda_1(1-\delta)\gs\frac{\pi^2}{d^2}.
\]
Hence, we get
\[
\lambda_1\gs\frac{\pi^2}{d^2}+\lambda_1\delta=\frac{\pi^2}{d^2}+\frac{(n-1)K}{2}.
\]
This is the required estimate. So far we complete the proof.
%Now we obtain the sharp estimate $\lambda_1\gs \pi^2/d^2$.
%This concludes the proof of the theorem.
%This implies \eqref{11-5-23-21-02}. So far we complete the proof
%of Theorem \ref{Zhong-Yang-1983 Main Theorem}.
\end{proof}

\vskip1cm

%Similarly, with the same arguments of [1], one can prove a version
%for a problem of case I.

\section{Further open problem}
At last we put forward a question
\begin{quote}
\emph{Does every compact Riemannian manifold have the symmetry that the minimum of the
first eigenfunction is the negative of the maximum, i.e., $k=1$ in \eqref{11-5-16-18-30}?}
\end{quote}
Indeed, this question was already implicit in \cite{Zhong-Yang-1983}.
Until now one do not know whether the first eigenfunction is symmetric or not.
Perhaps it is more difficult to solve this question
even if we add some assumption on the underlying Riemannian manifold.
It is well know that if one can give an affirmative answer to this difficult question,
or equivalently one can prove $k=1$, then also one can obtain easily the optimal estimate
\[
\lambda_1\gs \frac{\pi^2}{d^2}+\frac{(n-1)K}{2}
\]
by Theorem \ref{thm-Jun-Ling-2007-AGAG-2} and \ref{thm-12-9-27-8-6}.

%\vskip1cm

%\subsection*{Acknowledgment}
%%The author wishes to take this opportunity to thank the referee for his
%%%(or her) valuable advice and suggestion, and a lot of help.
%%(or her) valuable advice and suggestion.
%It is the author's pleasure to thank the referee for his
%(or her) valuable advice and suggestion.

%%%----------------------------------------------------------------------------

\end{document}